\documentclass[a4paper,reqno]{amsart}
\usepackage{amssymb, amsmath, amscd}

\def\XIx\langle#1\rangle{h(#1)}

\newtheorem{theorem}{Theorem}[section]

\newtheorem{lemma}[theorem]{Lemma}

\newtheorem{remark}[theorem]{Remark}

\def\qedbox{\hbox{$\rlap{$\sqcap$}\sqcup$}}
\def\id{\operatorname{id}}\def\Tr{\operatorname{Tr}}

\makeatletter
 \@addtoreset{equation}{section}
\makeatother
\begin{document}
\title
{K\"ahler--Weyl manifolds of dimension 4}
\author{P. Gilkey and S. Nik\v cevi\'c}
\address{PG: Math. Dept. Univ. Oregon, Eugene OR 97403 USA}
\email{gilkey@uoregon.edu}
\address{SN: Mathematical Institute, Sanu,
Knez Mihailova 36, p.p. 367,
11001 Belgrade,
Serbia}
\email{stanan@mi.sanu.ac.rs}
\begin{abstract}{We determine the space of algebraic
pseudo-Hermitian K\"ahler--Weyl curvature tensors and the space of para-Hermitian K\"ahler--Weyl curvature tensors in
dimension 4 and show that every algebraic possibility is geometrically
realizable. We establish the Gray identity for pseudo-Hermitian Weyl manifolds
and for para-Hermitian Weyl manifolds in arbitrary dimension. 
\\MSC 2002: 53B05, 15A72,
53A15, 53B10, 53C07, 53C25}\end{abstract}
\maketitle

\section{Introduction}
Let $(M,g)$ be a pseudo-Riemannian manifold of dimension $m=2\bar m\ge4$ with
$H^1(M;\mathbb{R})=0$.  Let $\nabla$ be a torsion-free connection on the
tangent bundle $TM$ of $M$. The triple $(M,g,\nabla)$ is said to be a {\it
Weyl structure} if $\nabla g=-2\phi\otimes g$ for some smooth $1$-form $\phi$
on $M$. Let $\nabla^g$ be the Levi-Civita connection of $g$ and let
$\phi^\star$ be the associated dual vector field. One has \cite{GNU10}:
\begin{equation}\label{eqn-1.a}
\nabla_xy:=\nabla_x^gy+\phi(x)y+\phi(y)x-g(x,y)\phi^\star\,.
\end{equation}

These geometries were first introduced by Weyl
\cite{W22} and remain an active area of investigation today -- see, for
example, the discussion in
\cite{G09,N09,P10,V10}. Weyl structures are intimately linked with conformal
geometry. If $\tilde g=e^{2f}g$ is a conformally equivalent metric, then
$(M,\tilde g,\nabla)$ is again a Weyl structure where $\tilde\phi=\phi-df$.
A Weyl structure is said to be {\it trivial} if $\phi=df$ for some smooth
function $f$ or, equivalently, if $\nabla=\nabla^{\tilde g}$ where
$\nabla^{\tilde g}$ is the Levi-Civita connection of the conformally
equivalent metric $\tilde g=e^{2f}$. Since we have assumed that
$H^1(M;\mathbb{R})=0$, the Weyl structure is trivial if and only if
$d\phi=0$.

Let $J_-$ (resp. $J_+$) be an almost complex (resp. para-complex) structure
on $TM$. It is convenient to use a common notation $J_\pm$ even though we
shall never be considering both structures simultaneously. One says that
$J_\pm$ is {\it integrable} if there exists a cover of $M$ by coordinate
charts $(x^1,...,x^{\bar m},y^1,...,y^{\bar m})$ so that
$$J_\pm:\partial_{x_i}\rightarrow\partial_{y_i}\quad\text{and}
\quad J_\pm:\partial_{y_i}\rightarrow\pm\partial_{x_i}\,.$$
We say that a torsion free connection $\nabla$ is {\it K\"ahler} if $\nabla
J_\pm=0$; the existence of such a connection then implies $J_\pm$ is
integrable. The triple $(M,g,J_\pm)$ is said to be a {\it
para/pseudo-Hermitian manifold} if $J_\pm^*g=\mp g$ and if $J_\pm$ is
integrable. If the Levi-Civita connection $\nabla^g$ is {\it K\"ahler}, then
$(M,g,J_\pm)$ is said to be {\it K\"ahler}.

We wish to study the interaction of these two structures. One says that a
quadruple $(M,g,J_\pm,\nabla)$ is a {\it K\"ahler--Weyl} structure if
$(M,g,J_\pm)$ is a para/pseudo-Hermitian manifold, if $(M,g,\nabla)$ is a
Weyl structure, and if $\nabla J_\pm=0$. The following is well known -- see,
for example, the discussion in \cite{PPS93} in the Riemannian setting (which
uses results of \cite{V82,V83}) and the generalization given in \cite{GS11}
to the more general context:

\begin{theorem}\label{thm-1.1} Let $m\ge6$. If $(M,g,J_\pm,\nabla)$ is a 
K\"ahler--Weyl structure, then the associated Weyl structure is trivial, i.e.
there is a conformally equivalent metric $\tilde g=e^{2f}g$ so that
$(M,\tilde g,J_\pm)$ is K\"ahler and so that $\nabla=\nabla^{\tilde g}$.
\end{theorem}

Examples in \cite{CP00,PS91} show that Theorem~\ref{thm-1.1}
fails if $m=4$ and motivate our present investigation. Let
$\Omega_\pm$ be the K\"ahler form:
$$\Omega_\pm(x,y):=g(x,J_\pm y)\,.$$
Let $d$ be the exterior derivative and let $\delta$ be the dual operator, the
interior coderivative. The {\it Lee form} is given, modulo a suitable
normalizing constant, by $J_\pm^*\delta\Omega_\pm$ and plays a crucial role.
The following result was established \cite{KK10} in the Riemannian
setting; the proof extends without change to this more general context:

\begin{theorem}\label{thm-1.2}
Every para/pseudo-Hermitian manifold of dimension $4$ admits a unique 
K\"ahler-Weyl structure where $\phi=\pm\frac12J_\pm^*\delta\Omega_\pm$.
\end{theorem}

The results of Theorem~\ref{thm-1.1} and of Theorem~\ref{thm-1.2} are
closely related to curvature decompositions. Let $R$ be the curvature tensor,
let $\mathcal{R}$ be the curvature operator, and let $\rho$ be the Ricci
tensor of a Weyl structure $(M,g,\nabla)$. They are defined by:
\begin{eqnarray*}
&&\mathcal{R}(x,y):=\nabla_x\nabla_y-\nabla_y\nabla_x-\nabla_{[x,y]},\\
&&R(x,y,z,w):=g(\mathcal{R}(x,y)z,w),\\
&&\rho(x,y):=\Tr\{z\rightarrow\mathcal{R}(z,x)y\}\,.
\end{eqnarray*}
Let $\rho_a(x,y):=\frac12\{\rho(x,y)-\rho(y,x)\}$ be the alternating part
of the Ricci tensor. The following facts are well known (see, for
example,
\cite{GI94,GNU10,PS91,PT93}):
\begin{equation}\label{eqn-1.b}
\begin{array}{l}
R(x,y,z,w)=-R(y,x,z,w),\\
R(x,y,z,w)+R(y,z,x,w)+R(z,x,y,w)=0,\vphantom{\vrule height 12pt}\\
R(x,y,z,w)+R(x,y,w,z)=-\textstyle\frac4m\rho_a(x,y)g(z,w)\vphantom{\vrule
height 12pt}\,.\end{array}\end{equation}
We also have the relation:
\begin{equation}\label{eqn-1.c}
 d\phi=\textstyle-\frac1m\rho_a\,.
\end{equation}

If $\nabla=\nabla^g$ is the Levi-Civita connection, then we have the
additional symmetry:
\begin{equation}\label{eqn-1.d}
R(x,y,z,w)+R(x,y,w,z)=0\,.
\end{equation}
The Weyl structure is trivial if and only if Equation~(\ref{eqn-1.d}) is
satisfied \cite{GNU10}. If $\nabla$ is K\"ahler,
then $\mathcal{R}(x,y)J_\pm=J_\pm\mathcal{R}(x,y)$ for all $x,y$ or,
equivalently:
\begin{equation}\label{eqn-1.e}
R(x,y,J_\pm z,J_\pm w)=\mp R(x,y,z,w)\,.
\end{equation}

We now pass to the algebraic context. Let $(V,\langle\cdot,\cdot\rangle)$ be
an inner product space. The space of {\it Weyl curvature tensors}
$\mathfrak{W}\subset\otimes^4V^*$ is defined by imposing the symmetry of
Equation~(\ref{eqn-1.b}). The space of {\it Riemann curvature tensors}
$\mathfrak{R}\subset\mathfrak{W}$ is obtained by requiring in addition the
symmetry of Equation~(\ref{eqn-1.d}). Let $J_\pm$ be a para/pseudo-Hermitian
structure on $(V,\langle\cdot,\cdot\rangle)$. We define the space of {\it
K\"ahler tensors} $\mathfrak{K}_\pm$ by imposing Equation~(\ref{eqn-1.e}).
The space of {\it K\"ahler--Weyl tensors}
$\mathfrak{K}_{\pm,\mathfrak{W}}:=\mathfrak{K}_\pm\cap\mathfrak{W}$ is
obtained by imposing the symmetries of Equation~(\ref{eqn-1.b}) and of
Equation~(\ref{eqn-1.e}) and the space of {\it K\"ahler--Riemann tensors}
$\mathfrak{K}_{\pm,\mathfrak{R}}:=\mathfrak{K}_\pm\cap\mathfrak{R}$ is
obtained by imposing in addition the symmetry of Equation~(\ref{eqn-1.d}).
The structure groups are given by:
\begin{eqnarray*}
&&\mathcal{O}:=\{T\in\operatorname{GL}:T^*\langle\cdot,\cdot\rangle
     =\langle\cdot,\cdot\rangle\},\\
&&\mathcal{U}_\pm:=\{T\in\mathcal{O}:TJ_\pm=J_\pm T\},\\
&&\mathcal{U}_\pm^\star:=\{T\in\mathcal{O}:TJ_\pm=J_\pm T\text{ or }
     TJ_\pm=-J_\pm T\}\,.
\end{eqnarray*}
It is convenient to work with the $\mathbb{Z}_2$ extensions
$\mathcal{U}_\pm^\star$ which permits us to interchange the roles of
$J_\pm$ and $-J_\pm$. Let $\chi$ be the $\mathbb{Z}_2$ valued character of
$\mathcal{U}_\pm^\star$ so that if $T\in\mathcal{U}_\pm^\star$, then $$J_\pm
T=\chi(T)TJ_\pm\,.$$ One then has that $T^*\Omega_\pm=\chi(T)\Omega_\pm$. Let
$$\Lambda_{0,J_\pm}^2=\{\Phi\in\Lambda^2(V^*):\Phi\perp\Omega_\pm\}\,.$$ 

\begin{theorem}\label{thm-1.3}
Let $(V,\langle\cdot,\cdot\rangle,J_\pm)$ be a para/pseudo-Hermitian vector
space.
\begin{enumerate}
\item If $m\ge6$, then
$\mathfrak{K}_{\pm,\mathfrak{W}}=\mathfrak{K}_{\pm,\mathfrak{R}}$.
\item If $m=4$, then
$\mathfrak{K}_{\pm,\mathfrak{W}}=\mathfrak{K}_{\pm,\mathfrak{R}}
   \oplus L_{0,J_\pm}$ where $L_{0,J_\pm}\approx\Lambda_{0,J_\pm}^2$ as a
$\mathcal{U}_\pm^\star$ module.
\end{enumerate}
\end{theorem} 

This is one of the facts about $4$-dimensional geometry that distinguishes it
from the higher dimensional setting; the module $L_{0,J_\pm}^2$ provides
additional curvature possibilities if
$m=4$. 

Let\ $(V,\langle\cdot,\cdot\rangle,J_\pm)$ be a para/pseudo-Hermitian vector
space and let
$A\in\mathfrak{K}_{\pm,\mathfrak{W}}$. We say that $A$ is {\it
geometrically realizable} if there exists a K\"ahler--Weyl structure
$(M,g,J_\pm,\nabla)$, $P\in M$, and an isomorphism $\phi:T_PM\rightarrow V$
so that
$\phi^*\langle\cdot,\cdot\rangle=g_P$, $\phi^*J_\pm=J_{\pm,P}$, and
$J^*A=R_P$.

\begin{theorem}\label{thm-1.4}
Every element of $\mathfrak{K}_{\pm,\mathfrak{W}}$ is
geometrically realizable.
\end{theorem}

Theorem~\ref{thm-1.4} means that Equation~(\ref{eqn-1.b}) and
Equation~(\ref{eqn-1.e}) generate the universal curvature symmetries of the
curvature tensor of a K\"ahler--Weyl structure; there are no hidden
symmetries. The fact that
$\mathfrak{K}_{\pm,\mathfrak{W}}\ne\mathfrak{K}_{\pm,\mathfrak{R}}$ in
dimension $4$ permits us to find K\"ahler--Weyl structures which do not
satisfy the symmetry of Equation~(\ref{eqn-1.d}) and which therefore are not
trivial. Thus it is the curvature decomposition of Theorem~\ref{thm-1.3}
which is at the heart of the difference between the $4$-dimensional setting
and the higher dimensional setting exemplified by Theorem~\ref{thm-1.1} and
by Theorem~\ref{thm-1.2}.

The {\it Gray symmetrizer} is
defined by setting:
\begin{eqnarray}
&&\mathcal{G}_\pm(A)(x,y,z,w):=A(x,y,z,w)+A(J_\pm x,J_\pm
y,J_\pm z,J_\pm w)\nonumber\\
&&\qquad\pm A(J_\pm x,J_\pm y,z,w)\pm A(x,y,J_\pm z,J_\pm w)\pm A(J_\pm
x,y,J_\pm z,w)\label{eqn-1.f}\\
&&\qquad\pm A(x,J_\pm y,z,J_\pm w)\pm A(J_\pm x,y,z,J_\pm w)\pm A(x,J_\pm
y,J_\pm z,w)\,.\nonumber
\end{eqnarray}
Gray \cite{gray} showed that the integrability of the (para)-complex
structure gives rise to the additional curvature identity
$\mathcal{G}(R^g)=0$. Although his result was originally stated only in the
Hermitian setting, it extends easily to the para/pseudo-Hermitian setting
\cite{BGKN09,BGNV09}. In fact, this identity remains valid in the context of
Weyl geometry:

\begin{theorem}\label{thm-1.5}
Let $(M,g,J_\pm)$ be a para/pseudo-Hermitian manifold and let $\nabla$ be a Weyl
connection. Then $\mathcal{G}(R^\nabla)=0$.
\end{theorem}

Here is a brief outline of this paper. In Section \ref{sect-2}, we review
some decomposition results that are needed. In Section
\ref{sect-3}, we establish Theorem \ref{thm-1.2}; we shall not follow the
discussion in \cite{KK10} but rather base our discussion on the
decomposition results of \cite{BGGH11,GH80} given in Theorem~\ref{thm-2.5} as
that will be more convenient for our further development. In Section
\ref{sect-4}, we prove Theorem \ref{thm-1.3}; we restrict to the case $m=4$
since the case
$m\ge6$ is treated in \cite{GS11}. We also verify Theorem \ref{thm-1.4}.
Since every element of
$\mathfrak{K}_{\pm,\mathfrak{R}}$ can be geometrically realized by a
para/pseudo-K\"ahler manifold \cite{BGM10}, Theorem~\ref{thm-1.4} follows
from Theorem~\ref{thm-1.3} if $m\ge6$. It therefore suffices to prove
Theorem~\ref{thm-1.4} if $m=4$. In Section~\ref{sect-5}, we use
Theorem~\ref{thm-1.3} to prove Theorem~\ref{thm-1.5}.

\section{Decomposition results}\label{sect-2}
In Section~\ref{sect-2.1}, we recall the fundamental facts of group
representation theory that we shall need; we work in the context of
$\mathcal{U}_\pm^\star$ modules as many of the relevant results fail for
$\mathcal{U}_+$. In Section~\ref{sect-2.2}, we review the Tricerri-Vanhecke
decomposition of
$\mathfrak{R}$ as a $\mathcal{U}_\pm^\star$ module. In
Section~\ref{sect-2.3}, we combine the Higa decomposition of $\mathfrak{W}$
with the Tricerri-Vanhecke decomposition to decompose $\mathfrak{W}$ as a
$\mathcal{U}_\pm^\star$ module. In Section~\ref{sect-2.4}, we present the
Gray-Hervella decomposition of the space of covariant derivatives of the
K\"ahler form as a $\mathcal{U}_\pm^\star$ module. 

\subsection{Representation Theory}\label{sect-2.1}
Let $(V,\langle\cdot,\cdot\rangle,J_\pm)$ be a para/pseudo-Hermitian
vector space. Extend $\langle\cdot,\cdot\rangle$ to a
non-degenerate symmetric bilinear form on $\otimes^kV^*$ by setting:
$$
\langle(v_1\otimes\dots\otimes v_k),(w_1\otimes\dots\otimes w_k)\rangle:=\prod_{i=1}^k\langle
v_i,w_i\rangle\,.
$$
Use $\langle\cdot,\cdot\rangle$ to identify $\otimes^kV$
with $\otimes^kV^*$ henceforth. The natural action of
$\mathcal{U}_\pm^\star$ on $\otimes^kV^*$ by pullback is an isometry making
any
$\mathcal{U}_\pm^\star$-invariant subspace of
$\otimes^kV^*$ into a $\mathcal{U}_\pm^\star$ module.
We refer to
\cite{BGGH11} for the proof of the following result; this result fails for
the group $\mathcal{U}_+$ and for that reason we choose to work with the
groups $\mathcal{U}_\pm^\star$.
\begin{lemma}\label{lem-2.1}
Let $(V,\langle\cdot,\cdot\rangle,J_\pm)$ be a para/pseudo-Hermitian vector
space. Let $\xi$ be a
$\mathcal{U}_\pm^\star$ submodule of
$\otimes^kV$.
\begin{enumerate}\item $\langle\cdot,\cdot\rangle$ is non-degenerate on $\xi$.
\item There is an orthogonal direct sum decomposition $\xi=\eta_1\oplus\dots\oplus\eta_k$ where the $\eta_i$
are irreducible $\mathcal{U}_\pm^\star$ modules.
\item If $\xi_1$ and $\xi_2$ are inequivalent irreducible $\mathcal{U}_\pm^\star$ submodules of $\xi$, then $\xi_1\perp\xi_2$.
\item The multiplicity with which an irreducible representation appears in $\xi$ is independent of the decomposition in
{\rm(2)}.
\item If $\xi_1$ appears with multiplicity $1$ in $\xi$ and if $\eta$ is any $\mathcal{U}_\pm^\star$ submodule of $\xi$, then
either
$\xi_1\subset\eta$ or else $\xi_1\perp\eta$.
\item If $0\rightarrow\xi_1\rightarrow\xi\rightarrow\xi_2\rightarrow0$ is a short exact sequence of
$\mathcal{U}_\pm^\star$ modules, then
$\xi\approx\xi_1\oplus\xi_2$ as a $\mathcal{U}_\pm^\star$ module.
\end{enumerate}
\end{lemma}

\subsection{The Tricerri-Vanhecke decomposition}\label{sect-2.2}
Decompose
$\otimes^2V^*=S^2\oplus\Lambda^2$ as the direct sum of the symmetric 
and of the alternating $2$-tensors, respectively. Set
\begin{eqnarray*}
&&S_{+,J_\pm}^{2}:=\{\theta\in S^2:J_\pm^*\theta=+\theta\},\qquad\qquad
\Lambda_{+,J_\pm}^{2}:=\{\theta\in\Lambda^2:J_\pm^*\theta=+\theta\}\\
&&S_{-,J_\pm}^{2}:=\{\theta\in S^2:J_\pm^*\theta=-\theta\},\qquad\qquad
\Lambda_{-,J_\pm}^{2}:=\{\theta\in\Lambda^2:J_\pm^*\theta=-\theta\}\,.
\end{eqnarray*}
We have $\langle\cdot,\cdot\rangle\in S_{\mp,J_\pm}^2$ and
$\Omega_\pm\in\Lambda_{\mp,J_\pm}^2$. This permits us to express
\begin{eqnarray*}
&&S_{\mp,J_\pm}^2=\langle\cdot,\cdot\rangle\cdot\mathbb{R}\oplus 
S_{0,\mp,J_\pm}^2\quad\text{and}\quad
\Lambda_{\mp,J_\pm}^2=\Omega_\pm\cdot\mathbb{R}\oplus\Lambda_{0,\mp,J_\pm}^2
\quad\text{where}\\
&&S_{0,\mp,J_\pm}^{2}:=\{\theta\in
S_{\mp,J_\pm}^2:\theta\perp\langle\cdot,\cdot\rangle\}\quad\text{and}\quad
\Lambda_{0,\mp,J_\pm}^{2}:=
\{\theta\in\Lambda_{\mp,J_\pm}^2:\theta\perp\Omega_\pm\}\,.
\end{eqnarray*}
This gives the following orthogonal decomposition of $\otimes^2V^*$ into
irreducible and inequivalent $\mathcal{U}_\pm^\star$ modules:
\begin{equation}\label{eqn-2.a}
\otimes^2V^*=S_{\pm,J_\pm}^2\oplus \mathbb{R}\oplus 
S_{0,\mp,J_\pm}^2\oplus
\Lambda_{\pm,J_\pm}^2\oplus\chi\oplus\Lambda_{0,\mp,J_\pm}^2\,.
\end{equation}

The following decompositions were first established in \cite{M73,S73,TV81}
for almost complex structures in the positive definite case; we refer to
\cite{BGN11} for the extension to the higher signature setting and to the
almost para-complex case:

\goodbreak\begin{theorem}\label{thm-2.2}
Adopt the notation established above. We have an orthogonal direct sum
decompositions of non-trivial irreducible $\mathcal{U}_\pm^\star$ modules:
\begin{eqnarray*}
&&\mathfrak{R}=
\left\{\begin{array}{ll}
\oplus_{1\le i\le 10}W_{\pm,i}&\text{ if }m\ge 8\\
\oplus_{1\le i\le 10,i\ne 6}W_{\pm,i}&\text{ if }m=6\\
\oplus_{1\le i\le 10,i\ne 5,6,10}W_{\pm,i}&
     \text{ if }m=4\end{array}\right\},\\
&&\mathfrak{K}_{\pm,\mathfrak{R}}=W_{\pm,1}\oplus
W_{\pm,2}\oplus W_{\pm,3}\,.
\end{eqnarray*}
These are inequivalent
$\mathcal{U}_\pm^\star$ modules except for the isomorphisms:
$$
W_{\pm,1}\approx
W_{\pm,4}\approx\mathbb{R}\quad\text{and}\quad W_{\pm,2}\approx
W_{\pm,5}\approx S_{0,\mp,J_\pm}^2\,.
$$
We have $W_{\pm,8}\approx
S_{\pm,J_\pm}^2$ and $W_{\pm,9}\approx\Lambda_{\pm,J_\pm}^2$ as
$\mathcal{U}_\pm^\star$ modules. None of the modules $W_{\pm,i}$ is
isomorphic either to
$\chi$ or to
$\Lambda_{0,\mp,J_\pm}^2$ as $\mathcal{U}_\pm^\star$ modules.
\end{theorem}

The precise nature of the modules $W_{\pm,i}$ for $i=3,6,7,10$
is not relevant and we refer to \cite{TV81} in the Riemannian
setting and to  \cite{BGN11} in the general setting for their precise
definition.

\subsection{The Higa decomposition}\label{sect-2.3}
We refer to \cite{H93,H94} for the proof of:
\begin{theorem}\label{thm-2.3}
There is an orthogonal direct sum decomposition
$\mathfrak{W}=\mathfrak{R}\oplus L$ where $L\approx\Lambda^2$ as an
$\mathcal{O}$ module.
\end{theorem}

Decomposing
$\Lambda^2=\Omega_\pm\cdot\mathbb{R}\oplus\Lambda_{0,\mp,J_\pm}^2
\oplus\Lambda_{\pm,J_\pm}^2$ as a $\mathcal{U}_\pm^\star$ module and
then applying Theorem~\ref{thm-2.2} and Theorem~\ref{thm-2.3} yields:

\begin{theorem}\label{thm-2.4}
Let $(V,\langle\cdot,\cdot\rangle,J_\pm)$ be a para/pseudo-Hermitian vector
space.  We  have an orthogonal direct sum decomposition of
$\mathcal{U}_\pm^\star$ modules:
$$\mathfrak{W}=
\left\{\begin{array}{ll}
\oplus_{1\le i\le 13}W_{\pm,i}&\text{ if }m\ge 8\\
\oplus_{1\le i\le 13,i\ne 6}W_{\pm,i}&\text{ if }m=6\\
\oplus_{1\le i\le 13,i\ne 5,6,10}W_{\pm,i}&\text{ if
}m=4\end{array}\right\}.$$ We have  $W_{\pm,11}\approx\chi$,
$W_{\pm,12}\approx\Lambda_{0,\mp,J_\pm}^2$, and
$W_{\pm,13}\approx\Lambda_{\pm,J_\pm}^2$ as $\mathcal{U}_\pm^\star$ modules.
These are non-trivial inequivalent $\mathcal{U}_\pm^\star$ modules except for
the isomorphisms:
$$
W_{\pm,1}\approx W_{\pm,4}\approx\mathbb{R},\quad
W_{\pm,2}\approx W_{\pm,5}\approx S_{0,\mp,J_\pm}^2,\quad
W_{\pm,9}\approx W_{\pm,13}\approx\Lambda_{\pm,J_\pm}^2\,.
$$
\end{theorem}

\subsection{The Gray-Hervella decomposition}\label{sect-2.4}
 We follow \cite{BGGH11,GH80}. 
We assume $J_\pm$ is integrable. The covariant derivative
$\nabla^g\Omega_\pm$ has the symmetries:
\begin{equation}\label{eqn-2.b}
\begin{array}{l}
(\nabla^g\Omega_\pm)(x,y;z)=-(\nabla^g\Omega_\pm)(y,x;z)
=\pm(\nabla^g\Omega_\pm)(J_\pm x,J_\pm y;z)
\\
\phantom{\nabla^g\Omega_\pm(x,y;z)...}=\mp
(\nabla^g\Omega_\pm)(x,J_\pm y;J_\pm z)\,.\vphantom{\vrule  height 12pt}
\end{array}
\end{equation}
Let $(V,\langle\cdot,\cdot\rangle,J_\pm)$ be a para/pseudo-Hermitian vector
space. Let $\varepsilon_{ij}:=\langle e_i,e_j\rangle$ where
$\{e_i\}$ is a basis for $V$. Let $\phi\in V^*$.  Let $H\in\otimes^3V^*$. Let
$U_{\pm}$ be the space of tensors satisfying Equation~(\ref{eqn-2.b}). Set
\begin{eqnarray*}
&&\sigma_\pm(\phi)(x,y;z):=\phi({J_\pm}
x)\langle y,z\rangle-\phi({J_\pm} y)\langle x,z\rangle
+\phi(x)\langle{J_\pm} y,z\rangle-\phi(y)\langle{J_\pm} x,z\rangle,\\
&&(\tau_1H)(x):=\varepsilon^{ij}H(x,e_i;e_j)\,.
\end{eqnarray*}
The map $\tau_1$ appears in elliptic operator theory. Let $\delta$ be
coderivative -- $\delta$ is the formal adjoint of the exterior derivative
$d$. If
$\Phi$ is a smooth
$2$-form, then
\begin{equation}\label{eqn-2.c}
\delta\Phi=\tau_1\nabla^g\Phi\,.
\end{equation}

One has (see, for example, the discussion
in
\cite{BGGH11}) that:
\begin{equation}\label{eqn-2.d}
\tau_1\sigma_\pm=(m-2)J_\pm^\star\,.
\end{equation}
Thus $\operatorname{Range}(\sigma_\pm)\perp\ker(\tau_1)$ and these are
$\mathcal{U}_\pm^\star$ modules. We therefore set:
$$U_{\pm,3}:=U_\pm\cap\ker(\tau_1)\quad\text{and}\quad
U_{\pm,4}:=\operatorname{Range}(\sigma_\pm)\,.
$$
The following result follows from a more general result of \cite{GH80} in
the Hermitian setting; we refer to \cite{BGGH11} for the extension to the
pseudo-Hermitian and the para-Hermitian settings:

\begin{theorem}\label{thm-2.5}
Let $(V,\langle\cdot,\cdot\rangle,J_\pm)$ be a para/pseudo-Hermitian vector
space. We have a direct sum orthogonal decomposition of $U_\pm$ into
non-trivial irreducible and inequivalent
$\mathcal{U}_\pm^\star$ modules
 in the form:
\begin{eqnarray*}
&&U_\pm=\left\{\begin{array}{ll}
U_{\pm,3}\oplus U_{\pm,4}&\text{ if }m\ge6\\
U_{\pm,4}&\text{ if }m=4\end{array}\right\}.
\end{eqnarray*}
\end{theorem}

\section{The proof of Theorem \ref{thm-1.2}}\label{sect-3}

We adopt the notation of Theorem~\ref{thm-2.5}. We begin by establishing the
following result which is of interest in its own right.

\begin{theorem}\label{thm-3.1} Let $(M,g,J_\pm)$ be a para/pseudo-Hermitian
manifold.
\begin{enumerate}
\item The following assertions are equivalent:
\begin{enumerate}
\smallbreak\item $\nabla^g\Omega_\pm\in U_{\pm,4}$ for all points of $M$.
\smallbreak\item There exists $\nabla$ so $(M,g,J_\pm,\nabla)$ is a
K\"ahler--Weyl structure.
\end{enumerate}
\smallbreak\item If $(M,g,J_\pm,\nabla)$ is a K\"ahler--Weyl structure, then
$\phi=\pm\frac1{m-2}J_\pm^*\delta\Omega_\pm$.
\end{enumerate}
\end{theorem}

\begin{remark}\rm
By Assertion (2) and by Equation~(\ref{eqn-1.a}), the connection in Assertion
(1b) is uniquely determined by $(M,g,J_\pm)$.
\end{remark}

\begin{proof}
We compute directly that:
\begin{eqnarray*}
&&(\nabla\Omega_\pm)(x,y;z)=zg(x,J_\pm y)-g(\nabla_zx,J_\pm y)
    -g(x,J_\pm\nabla_zy)\\
&&\quad=zg(x,J_\pm y)-g(\nabla_zx,J_\pm y)-g(x,\nabla_zJ_\pm y)
    +g(x,(\nabla_zJ_\pm)y)\\
&&\quad=(\nabla_zg)(x,J_\pm y)+g(x,(\nabla_zJ_\pm)y)\\
&&\quad=-2\phi(z)g(x,J_\pm y)+g(x,(\nabla_zJ_\pm)y)\,.
\end{eqnarray*}
We use Equation~(\ref{eqn-1.a}) and the definition of $\sigma_\pm$ to
compute that:
\begin{eqnarray*}
&&(\nabla\Omega_\pm)(x,y;z)=zg(x,J_\pm y)-g(\nabla_z^gx,J_\pm y)
     -g(x,J_\pm\nabla_z^gy)\\
&&\qquad-\phi(z)g(x,J_\pm y)-\phi(x)g(z,J_\pm y)+g(x,z)g(\phi^*,J_\pm y)\\
&&\qquad-\phi(z)g(x,J_\pm y)-\phi(y)g(x,J_\pm z)+g(y,z)g(x,J_\pm\phi^*)\\
&&\quad=(\nabla^g\Omega)(x,y;z)-2\phi(z)g(x,J_\pm
   y)-\sigma_\pm(\phi)(x,y;z)\,.
\end{eqnarray*}
\medbreak\noindent
This leads to the relation:
\begin{equation}\label{eqn-3.a}
(\nabla^g\Omega_\pm)(x,y;z)=(\sigma_\pm\phi)(x,y;z)+g(x,(\nabla_zJ_\pm)y)\,.
\end{equation}

Suppose that there exists a torsion free connection $\nabla$ so that $\nabla
g=-2\phi\otimes\phi$ and so that $\nabla J_\pm=0$. By
Equation~(\ref{eqn-3.a}),
$$\nabla^g\Omega_\pm\in\operatorname{Range}(\sigma_\pm)=U_{\pm,4}\,.$$
Consequently, Assertion (1b) implies Assertion (1a). Conversely, suppose that
there exists a $1$-form $\phi$ so
$\nabla^g\Omega_\pm=\sigma_\pm(\phi)$. By
Equation~(\ref{eqn-2.d}),
$$\phi=\textstyle\pm\frac1{m-2}J_\pm^*\tau_1\nabla^g\Omega_\pm\,.$$ 
Consequently, $\phi$ is smooth. Motivated by Equation (\ref{eqn-1.a}), we
define a connection
$\nabla$ by setting:
$$\nabla_xy:=\nabla_x^gy+\phi(x)y+\phi(y)x-g(x,y)\phi^\star\,.$$
Since
$\nabla_xy-\nabla_yx=\nabla_x^gy-\nabla_y^gx=[x,y]$, $\nabla$ is torsion free.
Furthermore,
\begin{eqnarray*}
(\nabla_xg)(y,z)&=&xg(y,z)-g(\nabla_xy,z)-g(y,\nabla_xz)\\
&=&xg(y,z)-g(\nabla_x^gy,z)-g(y,\nabla_x^gz)\\
&-&\phi(x)g(y,z)-\phi(y)g(x,z)+g(x,y)\phi(z)\\
&-&\phi(x)g(z,y)-\phi(z)g(x,y)+g(x,z)\phi(y)\\
&=&-2\phi(x)g(y,z)\,.
\end{eqnarray*}
This shows $\nabla g=-2\phi\otimes g$ so $(M,g,\nabla)$
is a Weyl structure. We apply Equation~(\ref{eqn-3.a}) to conclude $\nabla
J_\pm=0$ and thus $(M,g,J_\pm,\nabla)$ is a K\"ahler--Weyl structure. This
shows that Assertion (1a) implies Assertion (1b) and completes the proof of
Assertion (1).

If $(M,g,J_\pm,\nabla)$ is a K\"ahler--Weyl structure, then
$\nabla^g\Omega_\pm=\sigma_\pm\phi$ by
Equation~(\ref{eqn-3.a}). We use
Equation~(\ref{eqn-2.c}) and Equation~(\ref{eqn-2.d}) to compute:
\begin{eqnarray*}
&&\tau_1(\sigma_\pm(\phi))(x)=(m-2)(J_\pm^*\phi)(x)\\
&=&(\tau_1\nabla^g\Omega_\pm)(x)
  =(\varepsilon^{ij}\nabla^g\Omega_\pm)(x,e_i;e_j)
  =(\delta\Omega_\pm)(x)\,.
\end{eqnarray*}
This shows that $(m-2)J_\pm^\star\phi=\delta\Omega_\pm$.
Since $J_\pm^\star
J_\pm^\star=\pm\id$, Assertion (2) follows.\end{proof}

Let $m=4$.  By Theorem~\ref{thm-2.5},
$\nabla^g\Omega_\pm=\sigma_\pm(\phi)$ for some $\phi$. By
Theorem~\ref{thm-3.1}, 
$(M,g,J_\pm,\nabla)$ is a K\"ahler--Weyl structure where
$\phi=\pm\frac12J_\pm^*\delta\Omega_\pm$. This proves
Theorem~\ref{thm-1.2}.\hfill\qed

\section{The proof of Theorem \ref{thm-1.3} and of Theorem
\ref{thm-1.4}}\label{sect-4}

We begin with a simple example. Let $(x^1,x^2,x^3,x^4)$ be the usual
coordinates on $\mathbb{R}^4$. Define the canonical (para)-complex
structure $J_\pm$ on $\mathbb{R}^4$ by setting:
\begin{equation}\label{eqn-4.a}
J_\pm(\partial_{x_1})=\partial_{x_2},\quad
J_\pm(\partial_{x_2})=\pm\partial_{x_1},\quad
J_\pm(\partial_{x_3})=\partial_{x_4},\quad
J_\pm(\partial_{x_4})=\pm\partial_{x_3}\,.
\end{equation}
If $g$ is a para/pseudo-Hermitian metric on $\mathbb{R}^4$, set
$$
   g(\partial_{x_i},\partial_{x_j};\partial_{x_k})
  =\partial_{x_k}g(\partial_{x_i},\partial_{x_j})\,.
$$
We then have \cite{BGGH11}:
\begin{equation}\begin{array}{l}\label{eqn-4.b}
(\nabla^g\Omega_\pm)(\partial_{x_i},\partial_{x_j};\partial_{x_k})
     =\textstyle\frac12\{
g(\partial_{x_i},\partial_{x_k};J_\pm \partial_{x_j})
     -g(\partial_{x_j},\partial_{x_k};J_\pm \partial_{x_i})\\
\qquad\qquad\qquad\qquad\qquad\quad
     +g(J_\pm\partial_{x_i},\partial_{x_k};\partial_{x_j})
     -g(J_\pm \partial_{x_j},\partial_{x_k};\partial_{x_i})\}\,.
\vphantom{\vrule height 11pt}\end{array}\end{equation}
We consider a flat background metric
\begin{equation}\label{eqn-4.c}
g_0:=\varepsilon_{11}(dx^1\otimes dx^1\mp dx^2\otimes dx^2)
+\varepsilon_{22}(dx^3\otimes dx^3\mp dx^4\otimes dx^4)\,.
\end{equation}
We take $\varepsilon_{11}=\varepsilon_{22}=1$ to define a Hermitian metric,
$\varepsilon_{11}=1$ and $\varepsilon_{22}=-1$ to define a pseudo-Hermitian
metric of signature $(2,2)$, and $\varepsilon_{11}=\varepsilon_{22}=-1$ to
define a pseudo-Hermitian metric of signature $(4,0)$. We take
$\varepsilon_{11}=\varepsilon_{22}=1$ (and change the sign on
$\partial_{x_2}$ and $\partial_{x_4}$) to define a para-Hermitian metric.

\begin{lemma}\label{lem-4.1}
Let $f=f(x_1,x_3)$ be a smooth function on $\mathbb{R}^4$. Perturb the
metric of {\rm Equation~(\ref{eqn-4.c})} to define:
$$
g_f:=\varepsilon_{11}e^{2f}(dx^1\otimes dx^1\mp dx^2\otimes dx^2)
+\varepsilon_{22}(dx^3\otimes dx^3\mp dx^4\otimes dx^4)\,.
$$
This is a para/pseudo-Hermitian metric on $\mathbb{R}^4$. Apply
{\rm Theorem~\ref{thm-3.1}} to choose $\nabla$ so $(M,g_f,J_\pm,\nabla)$ is a
K\"ahler--Weyl structure. Then
$\rho_a=\pm4\partial_{x_1}\partial_{x_3}fdx^1\wedge dx^3$.
 \end{lemma}

\begin{proof}
We apply Equation~(\ref{eqn-4.b}) to see
$$
(\nabla^{g_f}\Omega_\pm)(\partial_{x_1},\partial_{x_3};\partial_{x_k})=
\left\{\begin{array}{rll}
\mp\varepsilon_{11}e^{2f}\partial_{x_3}f&\text{if}&k=2\\
0&\text{if}&k\ne2
\end{array}\right\}\,.$$
We apply Equation~(\ref{eqn-2.b}) to see that the non-zero components of
$\nabla^{g_f}\Omega_\pm$ are given, up to the
$\mathbb{Z}_2$ symmetry in the first components, by:
$$\begin{array}{ll}
(\nabla^{g_f}\Omega_\pm)(\partial_{x_1},\partial_{x_3};\partial_{x_2})
     =\mp\varepsilon_{11}e^{2f}\partial_{x_3}f,&
(\nabla^{g_f}\Omega_\pm)(\partial_{x_1},\partial_{x_4};\partial_{x_1})
     =\pm\varepsilon_{11}e^{2f}\partial_{x_3}f,\\
(\nabla^{g_f}\Omega_\pm)(\partial_{x_2},\partial_{x_4};\partial_{x_2})
     =-\varepsilon_{11}e^{2f}\partial_{x_3}f,&
(\nabla^{g_f}\Omega_\pm)(\partial_{x_2},\partial_{x_3};\partial_{x_1})
   =\pm\varepsilon_{11}e^{2f}\partial_{x_3}f.\vphantom{\vrule height
12pt}
\end{array}$$
We contract indices and apply Theorem~\ref{thm-3.1} to see:
$$
\phi=\textstyle\pm\frac12J_\pm^*\delta\Omega_\pm
=\pm\frac12J_\pm^*\tau_1(\nabla^{g_f}\Omega_\pm)
=\pm J_\pm^*\{\mp\partial_{x_3}f\cdot
dx^4\}=\mp\partial_{x_3}f\cdot dx^3\,.
$$
Since $f=f(x_1,x_3)$, the desired conclusion now follows from
Equation~(\ref{eqn-1.c}).
\end{proof}
\subsection{The proof of Theorem~\ref{thm-1.3}}
Let $m=4$. We apply Lemma~\ref{lem-2.1}, Theorem~\ref{thm-2.2}, and
Theorem~\ref{thm-2.4}. Let
$\xi$ be an irreducible
$\mathcal{U}_\pm^\star$ submodule of $\mathfrak{K}_{\pm,\mathfrak{W}}$. If
$\xi$ is not isomorphic to a submodule of
$\Lambda^2$, then $\xi$ must be a submodule of $\mathfrak{R}$ and hence
$$
  \xi\subset\mathfrak{R}\cap\mathfrak{K}_\pm
   =W_{\pm,1}\oplus W_{\pm,2}\oplus W_{\pm,3}\,.
$$
Since the modules $W_{\pm,i}$ are inequivalent and irreducible
for $i=1,2,3$, we have $\xi=W_{\pm,i}$ for $i=1,2,3$.

We therefore suppose that $\xi$ is isomorphic to a submodule of
$\Lambda^2$. If $\psi\in\Lambda^2$, set:
\begin{eqnarray*}
&&\Xi(\psi)(x,y,z,w):=2\psi(x,y)\langle z,w\rangle+\psi(x,z)\langle
y,w\rangle-\psi(y,z)\langle x,w\rangle\\
&&\phantom{\sigma(\psi)(x,y,z,w):}-\psi(x,w)\langle
y,z\rangle+\psi(y,w)\langle x,z\rangle\,.
\end{eqnarray*}
\medbreak\noindent We then have \cite{H93,H94,TV81} that the module $L$ of
Theorem~\ref{thm-2.2} is the image of $\Xi$.
Suppose that $\xi\approx\chi$. Then $\xi$ appears with
multiplicity $1$ and thus
$$\xi=W_{\pm,11}=\Xi(\Omega_\pm)\cdot\mathbb{R}\,.$$
Let $J_\pm$ be the (para)-complex structure on $\mathbb{R}^4$ given in
Equation~(\ref{eqn-4.a}) and let $g$ be the metric of
Equation~(\ref{eqn-4.c}). We show
$\Xi(\Omega_\pm)$ is not a K\"ahler tensor and thus $\xi\not\approx\chi$ by
computing:
\begin{eqnarray*}
&&\phantom{\mp}\Xi(\Omega_\pm)(e_1,e_4,e_3,e_1)=-g(e_4,J_\pm
e_3)g(e_1,e_1)=-g_{11}g_{44},\\
&&\mp\Xi(\Omega_\pm)(e_1,e_4,J_\pm e_3,J_\pm e_1)=\pm
g(e_1,J_\pm J_\pm e_1)g(e_4,J_\pm e_3)=g_{11}g_{44}\,.
\end{eqnarray*}

Let $f=\pm\frac14x_1x_3$. By Lemma~\ref{lem-4.1},
$\rho_a(f)=dx^1\wedge dx^3$; this is perpendicular to $\Omega_\pm$. Clearly
$\rho_a(f)$ has non-trivial components in both $\Lambda_{0,\mp,J_\pm}^2$ and
$\Lambda_{\pm,J_\pm}^2$. By Lemma~\ref{lem-2.1}, this means that both of the
modules $\Lambda_{\pm,J_\pm}^2$ and
$\Lambda_{0,\mp,J_\pm}^2$ appear with multiplicity at least 1 in
$\mathfrak{K}_{\pm,\mathfrak{W}}$. By Theorem~\ref{thm-2.4},
$\Lambda_{0,\mp,J_\pm}^2$ appears with multiplicity 1 in
$\mathfrak{W}$. Thus $\Lambda_{0,\mp,J_\pm}^2$ appears with multiplicity 1 in
$\mathfrak{K}_{\pm,\mathfrak{W}}$. Since $\Lambda_{\pm,J_\pm}^2$ appears with
multiplicity $2$ in $\mathfrak{W}$ and since
$W_{\pm,9}\approx\Lambda_{\pm,J_\pm}^2\subset\mathfrak{R}$ does not appear in
$\mathfrak{K}_{\mathfrak{R}}$, we conclude that $\Lambda_{\pm,J_\pm}^2$ appears with
multiplicity $1$ in $\mathfrak{K}_{\pm,\mathfrak{W}}$. Theorem~\ref{thm-1.3}
now follows.\hfill\qedbox
\subsection{The proof of Theorem~\ref{thm-1.4}}
Let $m=4$. Consider the space $\mathcal{S}$ of all germs of
para/pseudo-Hermitian metrics $g$ on $\mathbb{R}$ with the canonical
(para)-complex structure given in Equation~(\ref{eqn-4.a}) so that
$g(0)=g_0$ is the inner product of Equation~(\ref{eqn-4.c}).
$$g(0)=\varepsilon_{11}(dx^1\otimes dx^1\mp dx^2\otimes
dx^2)+\varepsilon_{22}(dx^3\otimes dx^3\mp dx^4\otimes dx^4)$$
and so that $dg(0)=0$. We let $\nabla$ be the associated K\"ahler--Weyl
connection and let $R=R(0)$. Let $\tilde{\mathfrak{K}}_{\pm,\mathfrak{W}}$
be the range of this map; this is $\mathcal{U}_\pm^\star$ module.
Results of \cite{BGM10} in the K\"ahler setting show every element of
$\mathfrak{K}_{\pm,\mathfrak{R}}$ can be geometrically realized by such a
K\"ahler metric; set $\nabla=\nabla^g$ to take the trivial Weyl structure.
Thus
$W_{\pm,i}\subset\tilde{\mathfrak{K}}_{\pm,\mathfrak{W}}$ for $i=1,2,3$.
Lemma~\ref{lem-4.1} shows $\tilde{\mathfrak{K}}_{\pm,\mathfrak{W}}$ contains
submodules isomorphic to $\Lambda_{\pm,J_\pm}^2$ and to $\Lambda_{0,\mp,J_\pm}^2$.
We may now apply Theorem~\ref{thm-1.3} to conclude
$\tilde{\mathfrak{K}}_{\pm,\mathfrak{W}}={\mathfrak{K}}_{\pm,\mathfrak{W}}$
and to complete the proof.\hfill\qedbox

\section{The proof of Theorem~\ref{thm-1.5}}\label{sect-5}
Let $\mathcal{G}_\pm$ be the Gray symmetrizer defined in
Equation~(\ref{eqn-1.f}). Then $\frac18\mathcal{G}_\pm$ is orthogonal
projection on the $\mathcal{U}_\pm^\star$ module $W_{\pm,7}$ appearing in
Theorem~\ref{thm-2.4}
\cite{BGN11,TV81}. Let
$(M,g,J_\pm)$ be a para/pseudo-Hermitian manifold and let $\nabla$ be a
torsion free connection with
$\nabla g=-2\phi\otimes g$. Choose $f\in
C^\infty(M)$ so that
$df(P)=\phi(P)$. If we replace $g$ by the conformally equivalent metric
$\tilde g=e^{2f}g$, then we replace $\phi$ by $\tilde\phi=\phi-df$. Thus
without loss of generality, we may assume that $\phi(P)=0$. The map
$\phi\rightarrow R^\nabla(P)-R^g(P)$ is then a linear map in the second
derivatives of $\phi$ and can be regarded as defining a map
$\Theta:\otimes^2T_P^*M\rightarrow\mathfrak{W}_P$. Since $W_{\pm,7}$
is not isomorphic to any $\mathcal{U}_\pm^\star$ submodule of
$\otimes^2T_P^*M$, we may apply Lemma~\ref{lem-2.1} to see that
$\mathcal{G}\circ\Theta=0$ and thus
$\mathcal{G}_\pm(R^\nabla)=\mathcal{G}_\pm(R^g)$. Since $J_\pm$ is
integrable, $\mathcal{G}_\pm(R^g)=0$ \cite{BGN11,gray}.
\hfill\qed

\section*{Acknowledgments} 
Research of P. Gilkey and S. Nik\v cevi\'c partially supported by Project
MTM2009-07756 (Spain), by DFG
PI 158/4-6 (Germany), and by Project 174012 (Serbia).

\end{document}